\documentclass{amsart} 

\title[Zig-zag chains and metric equivalences]{Zig-zag chains and metric equivalences between ultrametric spaces}

\author[Á.~Martínez-Pérez]{Álvaro Martínez-Pérez}
\address{Departamento de Geometría y Topología\\ Universidad Complutense de Madrid\\ Madrid 28040  Spain}
\email{alvaro\_martinez@mat.ucm.es}
       
\thanks{Partially supported by MTM 2006-00825.}

\usepackage[latin1]{inputenc}
\usepackage{amsfonts}
\usepackage{amsmath}
\usepackage{latexsym}
\usepackage[all]{xy}
\usepackage{amsthm}
\usepackage{graphicx}
\usepackage[T1]{fontenc}
\usepackage{hyperref}

\usepackage{caption}
\usepackage{verbatim}

\usepackage{amsmath}
\usepackage{amssymb}
\usepackage{amscd}
\usepackage{amsthm}
\usepackage{graphics}
\usepackage{graphpap}
\usepackage{pb-diagram}

\newtheorem{definicion}{Definition}[section]
\newtheorem{nota}[definicion]{Remark}
\newtheorem{prop}[definicion]{Proposition}
\newtheorem{lema}[definicion]{Lemma}

\newtheorem{teorema}[definicion]{Theorem}
\newtheorem{cor}[definicion]{Corollary}

\theoremstyle{definition} 

\newtheorem{notation}[definicion]{Notation}

\theoremstyle{remark}
\newtheorem{remark}[definicion]{Remark}

\newcommand{\br}{\ensuremath{\mathbb{R}}} 
\newcommand{\co}{\ensuremath{\colon}} 
\newcommand{\bz}{\ensuremath{\mathbb{Z}}} 
\newcommand{\bn}{\ensuremath{\mathbb{N}}} 

\begin{document}

\begin{abstract} We study the classification of ultrametric spaces based on their small scale geometry (uniform homeomorphism), large scale geometry (coarse equivalence) and both (all scale uniform equivalences). We prove that these equivalences can be characterized with parallel constructions using a combinatoric tool called \emph{common zig-zag chain}.
\end{abstract}

\maketitle
\tableofcontents

\begin{footnotesize}
\end{footnotesize}

\begin{footnotesize}
\end{footnotesize}

\section{Introduction}

When one defines continuity for a function on a metric space, one neglects a great deal of the information contained in the metric focusing on the small scale structure. In fact, if $d$ is a metric then so it is $d'=\min\{d,1\}$ and this change won't affect continuity nor the topology of the space.

The dual situation appears with bornologous maps, where we pay attention only to the large scale geometry. If we consider the metric $d'=\max\{d,1\}$ all the topology of the space is lost, but we still keep all its large scale properties. For a further development of this, see \cite{Roe1}.

Thus, uniform category and coarse category are partial and somehow dual aspects of the whole picture. To depict both scales we use \emph{all scale uniform maps}, which are both, uniformly continuous and bornologous.

In this paper, we consider three categories of ultrametric spaces. 

\begin{itemize}
\item $\mathcal{C}_1$: Complete ultrametric spaces and all scale uniform maps.
\item $\mathcal{C}_2$: Ultrametric spaces and surjective bornologous multi-maps.
\item $\mathcal{C}_3$: Complete ultrametric spaces and uniformly continuous maps.
\end{itemize}

The idea is to characterize equivalences in these three categories using the same combinatorial technique and the same arguments, presenting categories $\mathcal{C}_2$ and $\mathcal{C}_3$ as partial representations of the geometry in $\mathcal{C}_1$.

Let us recall here the definition and basic properties of ultrametric spaces.

\begin{definicion} If $(X,d)$ is a metric space and \ $d(x,y)\leq \max \{d(x,z),d(z,y)\}$
for all $x,y,z\in X$, then $d$ is an \emph{ultrametric} and
$(X,d)$ is an \emph{ultrametric space}.
\end{definicion}

\begin{lema}\label{ultrametrico} \begin{itemize}
\item[(a)] Any point of a ball is a center of the ball.
\item[(b)]If two balls have a common point, one is contained in
the other. \item[(c)] The diameter of a ball is less than or equal
to its radius. \item[(d)] In an ultrametric space, all triangles
are isosceles with at most one short side. \item[(e)]
$S_r(a)=\underset{x\in S_r(a)}{\cup}B_{<r}(x)$. \item[(f)] The
spheres $S_r(a) \ (r>0)$ are both open and closed. \hfill$\blacksquare$ 
\end{itemize} 
\end{lema}

There is a well known correspondence between ultrametric spaces and trees. In an ultrametric space, for any pair of intersecting balls one will contain the other and hence, considering partitions of the space with shrinking diameter we obtain a branching process which can be modelized by a tree.

In the bounded case, B. Hughes stablishes some categorical equivalences in \cite{Hug}, which capture the geometry of trees at infinity and local geometry of ultrametric spaces.

From a more topological point of view, in \cite{M-M}, it is proved that there is categorical equivalence between complete ultrametric spaces of diameter $\leq 1$ with uniformly continuous maps and rooted geodesically complete \br--trees with classes of rooted, continuous and metrically proper maps. Hence, uniform homeomorphisms between bounded ultrametric spaces, can be characterized by some kind of metrically proper homotopy equivalence between the trees. The technique to do this uses a function called \emph{modulus of continuity} which is associated to any uniformly continuous map. This idea is used here in a generalized way defining what will be called \emph{expansion function}.

In \cite{BZ} Taras Banakh and Ihor Zarichnyy characterize coarse equivalences of homogeneous ultrametric spaces by some intrinsic invariant of the spaces called sharp entropy. They do this using induction on partially ordered sets called towers. In a slightly different approach, these objects are treated here as chains instead of as ordered sets.

Trees are also related to chains and inverse sequences. In \cite{M-M2} it is proved an equivalence of categories between inverse sequences and rooted geodesically complete \br--trees oriented to a geometric description of the shape category in Marde\v{s}i\'c-Segal approach (see \cite{MS1}). In this paper, we defined a functor from maps between trees to morphisms of inverse sequences related to the construction used here.


A \emph{directed chain} $(X_k,\Phi_k)$ is a collection of sets $X_k \ k\in \bz$ and maps $\Phi_k:X_k \to X_{k+1} \ k\in \bz$. The \emph{direct limit}, $\underset{\to}{\lim} X_k$, is the disjoint union of the $X_k$'s modulo some equivalence relation $\sim$: for any pair of points $x_i\in X_i$, $x_j\in X_j$, $$x_i\sim x_j \mbox{ if there is some } k>i,j \mbox{ such that } f_{k-1}\circ \cdots \circ f_i (x_i)=f_{k-1}\circ \cdots \circ f_j(x_j).$$

\begin{definicion} A \emph{$D$--chain} $(X_k,\Phi_k)$ is a collection of sets $X_k \ k\in \bz$ and surjective maps $\Phi_k:X_k \to X_{k+1}, \ k\in \bz$, such that $\underset{\to}{\lim} X_k$ is trivial.
\end{definicion}

Hence, using all $k\in \bz$ we characterize the all scale uniform type of $U$, and if we want to focus only on the large scale or the small scale structure, we only need to restrict ourselves, roughly speaking, to one side of the chain. 

\begin{definicion} A \emph{$D_+$--chain} $(X_n,\Phi_n)$ is a sequence of sets $X_n$ and surjective maps $\Phi_n:X_n \to X_{n+1}$ such that $\underset{\to}{\lim} X_n$ is trivial.
\end{definicion}

\begin{definicion} A \emph{$D_-$--chain} $(X_n,p_n)$ is a sequence of sets $X_n$ and surjective maps $\Phi_n:X_{n+1} \to X_n$.
\end{definicion}

\begin{definicion} Given a $D$--chain $(X_k,\Phi_k)$ and an increasing function $\alpha\co \bz\to \bz$, the $D$--chain defined by the sets $X_{\alpha(k)}$ and the maps $\tilde{\phi}_k= \phi_{\alpha(k+1)-1}\circ \cdots \circ \Phi_{\alpha(k)}$ will be called an \emph{$\alpha$--sub-$D$-chain}.
\end{definicion}

\begin{definicion} Given a $D_+$--chain $(X_n,\Phi_n)$ and an increasing function $\alpha\co \bn\to \bn$, the $D_+$--chain defined by the sets $X_{\alpha(n)}$ and the maps $\tilde{\phi}_n= \phi_{\alpha(n+1)-1}\circ \cdots \circ \Phi_{\alpha(n)}$ will be called an \emph{$\alpha$--sub-$D_+$-chain}.
\end{definicion}

\begin{definicion} Given a $D_-$--chain $(X_n,\Phi_n)$ and an increasing function $\alpha\co \bn\to \bn$, the $D_-$--chain defined by the sets $X_{\alpha(n)}$ and the maps $\tilde{\phi}_n= \phi_{\alpha(n+1)-1}\circ \cdots \circ \Phi_{\alpha(n)}$ will be called an \emph{$\alpha$--sub-$D_-$-chain}.
\end{definicion}

\begin{remark} When there is no need to specify the map $\alpha$ and it is clear from the context whether we are considering $D$--chains, $D_+$--chains or $D_-$--chains, we will call these just \emph{subchains}.
\end{remark}

The following definitions are adapted from \cite{K}.

\begin{definicion} A $D$--chain $(Z_k,\mathcal{V}_k)$ is a \emph{common zig-zag $D$--chain} of the $D$--chains $(X_n,\Phi_n), (Y_n,\Psi_n)$ if there are increasing maps $\alpha,\beta \co \bz \to \bz$ and subchains $(X_{\alpha(k)},\tilde{\Phi}_k), (Y_{\beta(k)},\tilde{\Psi}_k)$ with
\[Z_i=\left\{
\begin{tabular}{l} $X_{\alpha(\frac{i+1}{2})} \quad
\qquad  \mbox{ if }
i$ is odd,\\

$Y_{\beta(\frac{i}{2})} \quad \qquad \quad \mbox{ if } i$ is even.\end{tabular}
 \right.\]
such that the following diagram commutes 
\begin{small}$$
\xymatrix{ \longrightarrow   & X_{\alpha(k-1)} \ar[dr]_{\mathcal{V}_{2k-3}} & \longrightarrow   &  X_{\alpha(k)}  \ar[dr]_{\mathcal{V}_{2k-1}}  &  \longrightarrow   & X_{\alpha(k+1)}\ar[dr]_{\mathcal{V}_{2k+1}}  & \longrightarrow \\
           \cdots & \longrightarrow     & Y_{\beta(k-1)} \ar[ur]_{\mathcal{V}_{2k-2}} & \longrightarrow &  Y_{\beta(k)} \ar[ur]_{\mathcal{V}_{2k}}& \longrightarrow  &  Y_{\beta(k+1)}}
$$\end{small}
\end{definicion}

\begin{definicion} A $D_+$--chain $(Z_n,\mathcal{V}_n)$ is a \emph{common zig-zag $D_+$--chain} of the $D_+$--chains $(X_n,\Phi_n), (Y_n,\Psi_n)$ if there are increasing maps $\alpha,\beta \co \bz_+ \to \bz_+$ and subchains $(X_{\alpha(n)},\tilde{\Phi}_n), (Y_{\beta(n)},\tilde{\Psi}_n)$ with
\[Z_i=\left\{
\begin{tabular}{l} $X_{\alpha(\frac{i+1}{2})} \quad
\qquad  \mbox{ if }
i$ is odd,\\

$Y_{\beta(\frac{i}{2})} \quad \qquad \quad \mbox{ if } i$ is even.\end{tabular}
 \right.\]
such that the following diagram commutes 
\begin{small}$$
\xymatrix{X_{\alpha(1)} \ar[dr]_{\mathcal{V}_{1}} & \longrightarrow   &  X_{\alpha(2)}  \ar[dr]_{\mathcal{V}_{3}}  &  \longrightarrow   & X_{\alpha(3)}\ar[dr]_{\mathcal{V}_{5}}  & \longrightarrow &\cdots \\
                & Y_{\beta(1)} \ar[ur]_{\mathcal{V}_{2}} & \longrightarrow &  Y_{\beta(2)} \ar[ur]_{\mathcal{V}_{4}}& \longrightarrow  &  Y_{\beta(3)} & \longrightarrow}
$$\end{small}
\end{definicion}

\begin{definicion} A $D_-$--chain $(Z_n,\mathcal{V}_n)$ is a \emph{common zig-zag $D_-$--chain} of the $D_-$--chains $(X_n,\Phi_n), (Y_n,\Psi_n)$ if there are increasing maps $\alpha,\beta \co \bn \to \bn$ and subchains $(X_{\alpha(n)},\tilde{\Phi}_n), (Y_{\beta(n)},\tilde{\Psi}_n)$ with
\[Z_i=\left\{
\begin{tabular}{l} $X_{\alpha(\frac{i+1}{2})} \quad
\qquad  \mbox{ if }
i$ is odd,\\

$Y_{\beta(\frac{i}{2})} \quad \qquad \quad \mbox{ if } i$ is even.\end{tabular}
 \right.\]
such that the following diagram commutes 
\begin{small}$$
\xymatrix{X_{\alpha(1)}  & \longleftarrow   &  X_{\alpha(2)}  \ar[dl]_{\mathcal{V}_{2}}  &  \longleftarrow   & X_{\alpha(3)}\ar[dl]_{\mathcal{V}_{4}}  & \longleftarrow &\cdots \\
                & Y_{\beta(1)} \ar[ul]_{\mathcal{V}_{1}} & \longleftarrow &  Y_{\beta(2)} \ar[ul]_{\mathcal{V}_{3}}& \longleftarrow  &  Y_{\beta(3)} \ar[ul]_{\mathcal{V}_{5}} & \longleftarrow}
$$\end{small}
\end{definicion}

The main result would be that two chains represent the same class of ultrametric space in the category $\mathcal{C}_1,\mathcal{C}_2,\mathcal{C}_3$ respectively if and only if there is a common zig-zag chain ($D$--chain, $D_+$--chain or $D_-$--chain) of them.

\begin{notation} We will denote by $(X_k,\Phi_k)\sim_{z-z} (Y_k,\Psi_k)$ if there is a common zig-zag chain of the chains $(X_k,\Phi_k), (Y_k,\Psi_k)$.
\end{notation}

This is related to some results in \cite{BZ}. In their work, they consider towers as ordered sets, which is just an alternative definition for what here is called $D_+$--chain. Also, we define here the end space of a chain and an ultrametric on it which is not exactly the same as they do.

Lemma 2 in \cite{BZ} states:

\begin{prop}\label{Admissible2} Let $\phi\co T_1 \to T_2$ be an admissible morphism between towers $T_1,T_2$. Then, the restriction $\Phi=\phi |_{[T_1]} \co [T_1]\to [T_2]$ is an asymorphism. \hfill$\blacksquare$
\end{prop}

Here we proof that there exists an admissible map if and only if there is a common zig-zag $D_+$--chain for the $D_+$--chains corresponding to the towers $T_1,T_2$ and that this implies a partial converse to \ref{Admissible2}.

\section{Expansion functions}

Let us recall first some definitions in coarse geometry. 

A map  between metric spaces $f\co X\to Y$ is \emph{metrically
proper} if for any bounded set $A\in Y$, \ $f^{-1}(A)$ is
bounded in $X$.

A map between metric spaces $f\co X \to Y$ is \emph{bornologous} if for every $R>0$ there is $S>0$
such that for any two points $x,x'\in X$ with $d(x,x')<R$,
$d(f(x),f(x'))<S$.

A map is \emph{coarse} if it is metrically proper
and bornologous.
 
Two maps between metric spaces $f,g\co X\to Y$ are \emph{close} if $\sup_{x\in X}\{d(f(x),g(x))<\infty$
 
A coarse map $f\co X \to Y$ is a \emph{coarse equivalence} if there is a coarse map $g\co Y \to X$ such that $g \circ f$ is close to $id_X$ and $f \circ g$ is close to $id_Y$. If there are such maps, then
$X,Y$ are \emph{coarse equivalent}.

But this in not the only way to define coarse equivalence between metric spaces. In this section, in order to describe in the same terms the categories $\mathcal{C}_1$, $\mathcal{C}_2$ and $\mathcal{C}_3$, we will use the following definition with multi-maps, as in \cite{BZ}.

By a multi-map $\Phi:X \Rightarrow Y$ between two sets $X,Y$ we understand any subset $\Phi \subset X\times Y$. For any subset $A\subset X$, by $\Phi(A)=\{y\in Y: \exists a\in A \mbox{ with } (a,y)\in \Phi \}$ we denote the image of $A$ under the multi-map $\Phi$. The inverse $\Phi^{-1}:Y\rightarrow X$ to the multi-map $\Phi$ is the subset $\Phi^{-1}=\{(y,x)\in Y\times X: (x,y)\in \Phi\}\subset Y\times X$. For two multi-maps $\Phi:X \Rightarrow Y$, $\Psi:Y \Rightarrow Z$ the composition $\Psi \circ \Phi$ is defined as usual: $$\Psi \circ \Phi =\{(x,z)\in X\times Z: \exists \, y\in Y \mbox{ such that } (x,y)\in \Phi \mbox{ and } (y,z)\in \Psi\}.$$

A multi-map is called \emph{surjective} if $\Phi(X)=Y$ and \emph{bijective} if $\Phi\subset X\times Y$ coincides with the graph of a bijective (single-valued) function.

\begin{definicion} Given a multi-map $\Phi\Rightarrow X \to Y$ between metric spaces, a non-decreasing function $\varrho_\Phi\co J\to [0,\infty)$ with $J=[0,S]$ or $J=[0,\infty)$ is called \emph{expansion function} if $\forall A \in X$ with $diam(A) \in J$, \
$diam(\Phi(A))\leq \varrho_\Phi(diam(A))$.
\end{definicion}

\begin{definicion} A multi-map $\Phi:X \Rightarrow Y$ between metric spaces is called 
\begin{itemize} \item \emph{bornologous} if there is an expansion function $\varrho_\Phi \co [0,\infty) \to [0,\infty)$.
\item an \emph{asymorphism} if both $\Phi,\Phi^{-1}$ are surjective bornologous multi-maps.
\end{itemize}
\end{definicion}

The following characterization is contained in Proposition 2 in \cite{BZ}.

\begin{prop}\label{coarse} For metric spaces $X,Y$ the following assertions are equivalent: 
\begin{itemize}\item $X$ and $Y$ are asymorphic. \item  $X$ and $Y$ are coarse equivalent. \hfill$\blacksquare$
\end{itemize} 
\end{prop}

\begin{nota} Thus, equivalences in $\mathcal{C}_2$ are, in fact, coarse equivalences of ultrametric spaces.
\end{nota}

\begin{definicion} $\Phi$ is called \emph{all scale uniform} if there is an expansion function $\varrho_\Phi \co [0,\infty) \to [0,\infty)$ such that $\varrho_\Phi(t)=0$ and $\lim_{t\to 0} \varrho_\Phi(t)=0$. In this case, since $\varrho_\Phi(t)=0$, $\Phi$ is a single-valued map. If $\Phi^{-1}$ is also all scale uniform we say that $X,Y$ are \emph{all scale uniform equivalent}. 
\end{definicion}

\begin{definicion} A map $f\co X\to Y$ between metric spaces is \emph{uniformly continuous} if $\forall \, \epsilon>0$ there exists some $\delta>0$ such that for any pair of points $x,y$ with $d_X(x,y)<\delta$ then $d_Y(f(x),f(y))<\epsilon$.
\end{definicion}

\begin{prop} A map $\Phi\co X\to Y$ between metric spaces is \emph{uniformly continuous} if and only if there is an expansion function $\varrho_\Phi \co [0,S] \to [0,\frac{1}{2}]$ such that $\varrho_\Phi(0)=0$ and $\lim_{t\to 0} \varrho_\Phi(t)=0$.
\end{prop}

\begin{proof} The if part is obvious.

If $\Phi$ is unifomly continuous here is some $S>0$ such that for any pair of points $x,y$ such that $d_X(x,x')<S$ then $d_Y(f(x),f(x'))<\frac{1}{2}$. Thus, it suffices to take $\varrho_\Phi(t):=\sup_{x,x'\in X,,\ d(x,x')\leq t}
\{d(\Phi(x),\Phi(x'))\}$.
\end{proof}

\section{All scale uniform equivalences}

Given an all scale uniform map $\Phi$ and an expansion function $\varrho_\Phi \co [0,\infty) \to [0,\infty)$ such that $\varrho_\Phi(t)=0$ and $\lim_{t\to 0} \varrho_\Phi(t)=0$, let us define $\gamma_{\varrho_\Phi}\co \bz \to \bz$ as follows, \begin{equation}\label{gamma} \gamma_{\varrho_\Phi}(k):=[log_2 (\varrho_\Phi(2^{k}))]+1\end{equation}
where $[t]$ stands for the maximal integer less or equal than $t$. Hence, for all points $x,x'\in X$, if $d_X(x,x')\leq 2^{k}$ then $d_Y(f(x),f(x'))\leq 2^{\gamma_{\varrho_\Phi}(k)}$ and $\gamma_{\varrho_\Phi}$ is non-decreasing. Then $\lim_{k\to -\infty}\gamma_{\varrho_\Phi}(k)=-\infty$ since $\Phi$ is uniformly continuous and we may assume, with no loss of generality, that $\lim_{t\to \infty}\varrho_\Phi(t)=\infty$ and therefore $\lim_{n\to \infty}\gamma_{\varrho_\Phi}(n)=\infty$.

\begin{nota} If $\Phi$ is an all scale uniform equivalence between unbounded metric spaces then necessarily $\lim_{t\to \infty}\varrho_\Phi(t)=\infty$ since $\Phi^{-1}$ is a bornologous surjective map.
\end{nota}

There is a correspondence between complete ultrametric spaces and $D$--chains. Let $U$ be an ultrametric space. For each $k\in \bz $ let $X_k$ be the partition of $U$ in balls of radius $2^{k}$. Let $\Phi_k\co X_k \to X_{k+1}$ the map canonically induced by the inclusion for any $k\in \bz$. $(X_k,\Phi_k)$ will be called the \emph{extended chain associated to $U$}. Conversely, given a $D$--chain we can obtain an ultrametric space as follows.

Let us define the end space as follows:

$$end(X_k,\Phi_k):=\{(x_k)_{k\in \bz} \ | \ x_k\in X_k \mbox{ and } \Phi(x_k)=x_{k+1}\},$$ and let us define the metric $$D((x_k),(y_k))=2^{k_0} \mbox{ where } k_0=\min\{k: x_k=y_k\}.$$ $D$ is well defined since $\underset{\to}{\lim}(X_k,\Phi_k)$ is trivial, and clearly, $D$ is an ultrametric.

\begin{prop}\label{end} If $(U,d)$ is a complete ultrametric space and $(X_k,\Phi_k)$ is the $D$--chain associated to $U$, then $(U,d)$ and $(end(X_k,\Phi_k),D)$ are bi-Lipschitz equivalent. In particular, they are all scale uniform equivalent.
\end{prop}

\begin{proof} First, note that if $(U,d)$ is complete there is a bijection $i\co U\to end(X_k,\Phi_k)$.

By the properties of the ultrametric, $d(x,y)\leq 2^k$ if and only if the points are in the same ball in the partition $X_k$. Hence, $d(x,y)\leq D(i(x),i(y))\leq 2d(x,y)$.
\end{proof}

\begin{definicion} A \emph{morphism of $D$--chains} $(f_k,\sigma) \co (X_k,\Phi_k)\to (Y_k,\Psi_k)$ consists of a non-decreasing function $\sigma\co \bz \to \bz$ such that $\lim_{k\to -\infty}=-\infty$ and $\lim_{k\to \infty}=\infty$, and maps $f_k\co X_k \to Y_{\sigma(k)}$ such that the following diagram commutes:
\begin{small}$$
\xymatrix{X_{k-1} \ar[dr]_{f_{k-1}} & \longrightarrow   &  X_{k}  \ar[dr]_{f_{k}}  &  \longrightarrow   & X_{k+1}\ar[dr]_{f_{k+1}}  & \longrightarrow \\
          \longrightarrow     & Y_{\sigma(k-1)} \longrightarrow & \cdots  & \longrightarrow Y_{\sigma(k)} \longrightarrow & \cdots  & \longrightarrow  Y_{\sigma(k+1)}}
$$\end{small}
\end{definicion}


\begin{nota} Notice that although in this definition $\sigma(k)$ and $\sigma(k+1)$ may be the same for some $k$, using that $\lim_{k\to -\infty}=-\infty$ and $\lim_{k\to \infty}=\infty$ then for some function $\alpha$, we may assume that $\sigma$ is increasing when restricted to the  $\alpha$-subchain.
\end{nota}

\begin{lema} \label{telescope} If $(X_k,\Phi_k)$ is a $D$--chain, $\alpha \co \bz \to \bz$ is an increasing map  and $(X_{\alpha(k)},\tilde{\Phi}_k)$ is the $\alpha$-subchain, then $end(X_k,\Phi_k)$ is all scale uniform equivalent to $end(X_{\alpha(k)},\tilde{\Phi}_k)$.
\end{lema}

\begin{proof} Consider the canonical map $i\co end(X_k,\Phi_k)\to end(X_{\alpha(k)},\tilde{\Phi}_k)$. 

Let us define the function $\Lambda\co \bz \to \bz$ such that $\Lambda(z)=\min\{k:\alpha(k)\geq z\}$. Since $\alpha$ is increasing, it follows that $\Lambda$ is non-decreasing, $\lim_{z\to -\infty}\Lambda(z)=-\infty$ and $\lim_{z\to \infty}\Lambda(z)=\infty$. Now, consider the function $\Gamma\co [0,\infty) \to [0,\infty)$ such that $\Gamma(0)=0$ and for any $x\in (2^{k-1},2^{k}]$, $\Gamma(x)=2^{\Lambda(k)}$. Clearly $\Gamma$ is non-decreasing, $\lim_{x\to 0}\Gamma(x)=0$ and $\lim_{x\to \infty}\Gamma(x)=\infty$. For any two end points $(x_k),(y_k)\in end(X_k,\Phi_k)$ with $D((x_k),(y_k))=2^{k_0}$ the distance between their correspondent subsequences $i(x_k)=(x_{\alpha(k)}),i(y_k)=(y_{\alpha(k)})$ is exactly $D'((x_{\alpha(k)}),(y_{\alpha(k)}))=\Gamma(2^{k_0})=2^{\Lambda(k_0)}$. Then $\Gamma$ is an extension function for $i\co end(X_k,\Phi_k) \to end((X_{\alpha(k)},\tilde{\Phi}_k))$ and $i$ is an all scale uniform map. A similar argument works for $i^{-1}$.
\end{proof}

\begin{prop} Consider two complete ultrametric spaces $U_1,U_2$ and let $(X_k,\Phi_k)$, $(Y_k,\Psi_k)$ be their associated $D$--chains. Then, there is an all scale uniform map $f\co U_1\to U_2$ if and only if there is a morphism of $D$--chains $(f_k,\sigma) \co (X_k,\Phi_k)\to (Y_k,\Psi_k)$.
\end{prop}

\begin{proof} If there is an all scale uniform map $\Phi:U_1\to U_2$, consider the map $\gamma_{\varrho_\Phi}\co \bz \to \bz$ from (\ref{gamma}). Let us consider $\sigma=\gamma_{\varrho_\Phi}$. Then, $\Phi$ induces maps $f_k\co X_k\to Y_{\sigma(k)}$ canonically as follows: by construction, any point $x_k\in X_k$ represents a ball $B(x_k)$ of radius $2^k$ of $U_1$, and by the properties of the ultrametric, this ball has diameter less or equal than $2^k$. By the definition of $\gamma_{\varrho_\Phi}$, if $diam(B(x_k))\leq 2^k$ then $diam(\Phi(B(x_k)))\leq 2^{\gamma_{\varrho_\Phi}(k)}=2^{\sigma(k)}$ and, since $Y_{\sigma(k)}$ is the partition of $U_2$ in balls of radius $2^{\sigma(k)}$, there is a unique point $y_{\sigma(k)}\in Y_{\sigma(k)}$ such that $\Phi(B(x_k))\subset B(y_{\sigma(k)})$. Then, the map $f_k$ such that $f_k(x_k):=y_{\sigma(k)}$ is well defined and it is surjective because $\Phi$ is surjective. It is immediate to check that the diagram commutes.

The morphism $(f_k,\sigma)$ induces a map $\Phi \co end(X_k,\Phi_k)\to end(Y_k,\Psi_k)$ where $\Phi((x_k))$ is the unique sequence $(y_k)\in end(Y_k,\Psi_k)$ such that $y_{\sigma(k)}=f_k(x_k)$. Now, for any $t\in (2^{k-1},2^k]$, $k\in\bz$, let $\varrho_\Phi(t)=2^{\sigma(k)}$ and $\varrho_\Phi(0)=0$. It is readily seen that $\varrho_\Phi$ is an extension function and $\Phi$ is an all scale uniform map. From Lemma \ref{telescope} together with Proposition \ref{end}, it follows that there is an all scale unifom map $f:U_1\to U_2$.
\end{proof}

\begin{lema}\label{alpha} If $\Phi \co X\to Y$ is an all scale uniform equivalence then there are increasing maps $\alpha(\Phi),\beta(\Phi^{-1}) \co \bz \to \bz$ such that $\gamma_{\varrho_\Phi}(\alpha(i))\leq \beta(i)$, $\gamma_{\varrho_{\Phi^{-1}}}(\beta(i))\leq \alpha(i+1)$ for every $i\in \bz$.
\end{lema}

\begin{proof} First, let $\alpha(0)=0$ and $\beta(0)=\gamma_\Phi(0)$. 

If we have defined $\alpha(i-1),\beta(i-1)$ for any $i>0$ then, it suffices make $\alpha(i)= \max\{\alpha(i-1)+1,\gamma_{\Phi^{-1}}(\beta(i-1))\}$ and $\beta(i)= \max\{\beta(i-1)+1,\gamma_{\Phi}(\alpha(i))\}$.

Since $lim_{k\to -\infty}\gamma_{\Phi^{-1}}(k)=-\infty$ for any $\alpha(i+1)$ there exist some $k_{\alpha(i+1)}$ such that for any $k\leq k_{\alpha(i+1)}$, $\gamma_{\Phi^{-1}}(k)\leq \alpha(i+1)$.

Since $lim_{k\to -\infty}\gamma_\Phi(k)=-\infty$ for any $\beta(i)$ there exist some $k_{\beta(i)}$  such that for any $k\leq k_{\beta(i)}$, $\gamma_{\Phi}(k)\leq \beta(i)$.

If we have defined $\alpha(i+1),\beta(i+1)$ for any $i<0$ then, it suffices make $\beta(i)= \min\{\beta(i+1)-1,k_{\alpha(i+1)}\}$ and $\alpha(i)=\min\{\alpha(i+1)-1,k_{\beta(i)}\}$.
\end{proof}

\begin{prop}\label{carac1} If  $(X_k,\Phi_k),(Y_k,\Psi_k)$ are two $D$--chains, then $(X_k,\Phi_k)\sim_{z-z}(Y_k,\Psi_k)$ if and only if $end(X_k,\Phi_k)$ and $end(Y_k,\Psi_k)$ are all scale uniform
equivalent.
\end{prop}

\begin{proof} If $(Z_k,\phi_k)$ is a common zig-zag $D$--chain, it suffices to check that $end(X_k\Phi_k))$ and $end(Y_k,\Psi_k)$ are all scale uniform equivalent to $end(Z_k,\phi_k)$ and this follows immediately from Lemma \ref{telescope}.

Now, suppose that there is an all scale unifom equivalence $\Phi\co end(X_k,\Phi_k) \to end(Y_k,\Psi_k)$. By Lemma \ref{alpha}, there are increasing maps $\alpha(\Phi),\beta(\Phi^{-1}) \co \bz \to \bz$ such that $\gamma_{\varrho_\Phi}(\alpha(i))\leq \beta(i)$, $\gamma_{\varrho_{\Phi^{-1}}}(\beta(i))\leq \alpha_{i+1}$ for every $i\in \bz$.

Therefore, $\Phi$ and $\Phi^{-1}$ canonically induce unique surjective maps $\mathcal{V}_{2i-1}\co X_{\alpha(i)}\to Y_{\beta(i)}$ and $\mathcal{V}_{2i}\co Y_{\beta(i)}\to X_{\alpha(i+1)}$. Since $\Phi$ is a bijection, $\mathcal{V}_{k+1} \circ \mathcal{V}_{k}$ coincides with the map induced by inclusion and therefore, making $Z_{2i-1}=X_{\alpha(i)}$ and $Z_{2i}=Y_{\beta(i)}$, $(Z_k,\mathcal{V}_k)$ is a common zig-zag $D$--chain of $(X_k,\Phi_k),(Y_k,\Psi_k)$.
\end{proof}

From this, and Proposition \ref{end}, it follows:

\begin{teorema}\label{caract1} Two complete ultrametric spaces are all scale uniform equivalent if and only if there is a common zig-zag $D$--chain between their associated $D$--chains. \hfill$\blacksquare$ 
\end{teorema}

Let $(X_k,\Phi_k),(Y_k,\Psi_k)$ be two $D$--chains and $(Z_k,\mathcal{V}_k)$ a common zig-zag $D$--chain with increasing maps $\alpha,\beta\co \bz \to \bz$ such that $Z_{2k-1}=X_{\alpha(k)}$ and $Z_{2k}=Y_{\beta(k)}$. Let us define $$f_Z\co  end(X_{\alpha(k)},\tilde{\Phi}_k) \to end(Y_{\beta(k)},\tilde{\Psi}_k)$$ such that for any end point $(x_{\alpha(k)})\in end(X_{\alpha(k)},\tilde{\Phi}_k)$, $f_Z((x_{\alpha(k)}))=(\mathcal{V}_{2k-1}(x_{\alpha(k)}))\in end(Y_{\beta(k)},\tilde{\Psi}_k)$. 

Let us recall that a function between metric spaces $f\co X\to Y$ is called \emph{bi-Lipschitz} if there is a constant $K>0$ such that for any pair of points $x,x'\in X$, $\frac{1}{K} \cdot d_X(x,x') \leq d_y(f(x),f(x')) \leq K\cdot d_X(x,x')$. If there is such a map, we say that $X,Y$ are bi-Lipschitz equivalent.

\begin{prop}\label{lipschitz} Given two $D$--chains $(X_k,\Phi_k),(Y_k,\Psi_k)$, then their end spaces $end(X_k,\Phi_k)$ and $end(Y_k,\Psi_k)$ are all scale uniform equivalent if and only if there are increasing sequences $\alpha,\beta\co \bz\to \bz$ such that $end(X_{\alpha(k)},\tilde{\Phi}_k)$ and $end(Y_{\beta(k)},\tilde{\Psi}_k)$ are bi-Lipschitz equivalent.
\end{prop}

\begin{proof} If $end(X_{\alpha(k)},\tilde{\Phi}_k)$ and $end(Y_{\beta(k)},\tilde{\Psi}_k)$ are bi-Lipschitz equivalent, then they are, in particular, all scale equivalent and so they are $end(X_n,\Phi_n)$ and $end(Y_n,\Psi_n)$ by Lemma \ref{telescope}.  

If $end(X_k,\Phi_k)$ and $end(Y_k,\Psi_k)$ are all scale equivalent then, by Proposition \ref{carac1}, there is a zig-zag common $D$--chain $(Z_k,\mathcal{V}_k)$ defined by sequences $\alpha,\beta$. The map $f_Z\co end(X_{\alpha(k)},\tilde{\Phi}_k) \to end(Y_{\beta(k)},\tilde{\Psi}_k)$ defined above holds that for any pair of end points $(x_{\alpha(k)}),(y_{\alpha(k)})\in end(X_{\alpha(k)},\tilde{\Phi}_k)$, $D_{T(\beta)}(f_Z((x_{\alpha(k)})),f_Z((y_{\alpha(k)}))) \leq  D_{T(\alpha)}(x_{\alpha(k)},y_{\alpha(k)}) \leq 2\cdot D_{T(\beta)}(f_Z((x_{\alpha(k)})),f_Z((y_{\alpha(k)})))$.
\end{proof}

\begin{cor}\label{cor1} Two ultrametric spaces $U_1,U_2$ are asymorphic if and only if there are increasing sequences $\alpha,\beta\co \bn\to \bn$ such that, for $(X_k,\Phi_k),(Y_k,\Psi_k)$ their associated $D$--chains, $end(X_{\alpha(n)},\tilde{\Phi}_n)$ and $end(Y_{\beta(n)},\tilde{\Psi}_n)$ are bi-Lipschitz equivalent. \hfill$\blacksquare$
\end{cor}

This can be translated into relations between ultrametric spaces avoiding $D$--chains. Proposition 2.2 in \cite{B-Al} states

\begin{prop} Let $(X,d)$ be a metric space. The metric $d$ is an ultrametric if and only if $f(d)$ is a metric for evey nondecreasing function $f\co \br_+ \to \br_+$. \hfill$\blacksquare$
\end{prop}

In particular, the new metric is also an ultrametric. Given an ultrametric space $(U,d)$ and a non-decreasing map $f\co \br_+ \to \br_+$, let us denote this new ultrametric as $d_f$, where $d_{f}(x,y):=f(d(x,y))$. 

Let $(U,d)$ be an ultrametric space and an increasing map $\gamma \co \bz \to \bz$. Let us define $f_{\gamma} \co \br_+ \to \bn$ a non-decreasing function such that for any $t\in (2^{\gamma(k-1)},2^{\gamma(k)}]$ $f(t)=2^{k}$ for every $k$. Let us denote simply by $(U,d(\gamma))$ the ultrametric space $(U,d_{f_{\gamma}})$ which depends only on the original ultrametric and $\gamma$.

Thus, from Corolary \ref{cor1} we obtain that,

\begin{cor} Two ultrametric spaces $(U_1,d_1),(U_2,d_2)$ are asymorphic if and only if there are increasing maps $\gamma_1,\gamma_2\co \bz \to \bz$ such that  
$(U_1,d_1(\gamma_1))$ and $(U_2,d_2(\gamma_2)$ are bi-Lipschitz equivalent. \hfill$\blacksquare$
\end{cor}

\section{Coarse equivalences}

In this section we treat the category $\mathcal{C}_2$. All we do, is to consider only the right side of the chain in the previous section and adapt the construction in some technical details.

Given a bornologous multimap $\Phi$ and its expansion function $\varrho_\Phi$, let us define $\gamma_{\varrho_\Phi}\co \bn \to \bn$ as follows, \begin{equation}\label{gamma2} \gamma_{\varrho_\Phi}(n):=[log_2 (\varrho_\Phi(2^{n}))]+1.\end{equation} Hence, for all points $x,x'\in X$, if $d_X(x,x')\leq 2^{n}$ then $d_Y(f(x),f(x'))\leq 2^{\gamma_{\varrho_\Phi}(n)}$ and $\gamma_{\varrho_\Phi}$ is non-decreasing. We may assume, with no loss of generality, that $\lim_{t\to \infty}\varrho_\Phi(t)=\infty$ and therefore $\lim_{n\to \infty}\gamma_{\varrho_\Phi}(n)=\infty$. 

\begin{nota} If $\Phi$ is an asymorphism between unbounded metric spaces then necessarily $\lim_{t\to \infty}\varrho_\Phi(t)=\infty$ since $\Phi^{-1}$ is a bornologous surjective map.
\end{nota}

There is a correspondence between ultrametric spaces and $D_+$--chains. For each $n\in \bn $ let $X_n$ be the partition of $U$ in closed balls of radius $2^{n}$. For each $x_n\in X_n$ let us denote by $B(x_n)$ the associated closed ball in $U$. Let $\Phi_n\co X_n \to X_{n+1}$ the map canonically induced by the inclusion for any $n\in \bn$. $(X_n,\Phi_n)$ will be called the \emph{$D_+$--chain associated to $U$}. Conversely, given a $D_+$--chain we can obtain an ultrametric space as follows.

The end space is:

$$end(X_n,\Phi_n):=\{(x_n)_{n\in \bn} \ | \ x_n\in X_n \mbox{ and } \Phi(x_n)=x_{n+1}\},$$ and the metric $$D((x_n),(y_n))=2^{n_0} \mbox{ where } n_0=\min\{n: x_n=y_n\}.$$ 

$D$ is well defined since $\underset{\to}{\lim}(X_n,\Phi_n)$ is trivial and $(end(X_n,\Phi_n),D)$ is an ultrametric space.

\begin{prop}\label{end2} If $U$ is an ultrametric space and $(X_n,\Phi_n)$ is the $D_+$--chain associated to $U$, then $U$ and $end(X_n,\Phi_n)$ are asymorphic (i.e., coarse equivalent). 
\end{prop}

\begin{proof} Consider the multi-map $\Phi\co U\Rightarrow end(X_n,\Phi_n)$ where $\Phi:=\{(x,(x_n)_{n\in \bn})| x\in B(x_1)\}$. Thus, if $diam(A)\leq 2$ then there exists some $x_1$ such that $A\subset B(x_1)$ and $diam(\Phi(A))=0$. If $2^{n}< diam(A)\leq 2^{n+1}$, then for each $x_n\in X_n$, $A\not \subset B(x_n)$ and, by the properties of the ultrametric, there is a unique $x_{n+1}\in X_{n+1}$ such that $A\subset B(x_{n+1})$ and therefore, $diam(\Phi(A))=2^{n+1}$ for every $n> 1$. Hence, $\Phi$ is an asymorphism.
\end{proof}

\begin{definicion} A \emph{morphism of $D_+$-chains} $(f_n,\sigma) \co (X_n,\Phi_n)\to (Y_n,\Psi_n)$ consists of a non-decreasing function $\sigma\co \bn \to \bn$ such that $\lim_{n\to \infty}=\infty$, and maps $f_n\co X_n \to Y_{\sigma(n)}$ such that the following diagram commutes:
\begin{small}$$
\xymatrix{X_{1} \ar[dr]_{f_{1}} & \longrightarrow   &  X_{2}  \ar[dr]_{f_{2}}  &  \longrightarrow   & X_{3}\ar[dr]_{f_{3}}  & \longrightarrow \\
               & Y_{\sigma(1)} \longrightarrow & \cdots  & \longrightarrow Y_{\sigma(2)} \longrightarrow & \cdots  & \longrightarrow  Y_{\sigma(3)}}
$$\end{small}
\end{definicion}

\begin{lema} \label{telescope2} If $(X_n,\Phi_n)$ is a $D_+$chain, $\alpha \co \bn \to \bn$ is an increasing map  and $(X_{\alpha(n)},\tilde{\Phi}_n)$ is the $\alpha$-subchain, then $end(X_n,\Phi_n)$ is asymorphic to $end(X_{\alpha(n)},\tilde{\Phi}_n)$.
\end{lema}

\begin{proof} There is a canonical map $i\co end(X_n,\Phi_n)\to end(X_{\alpha(n)},\tilde{\Phi}_n)$ with $i((x_n))=(x_{\alpha(n)})$. 

Let us define the function $\Lambda\co \bn \to \bn$ such that $\Lambda(n)=\min\{k:\alpha(k)\geq n\}$. Since $\alpha$ is increasing, it follows that $\Lambda$ is non-decreasing and $\lim_{n\to \infty}\Lambda(n)=\infty$. Now, consider the function $\Gamma\co [0,\infty) \to [0,\infty)$ such that $\Gamma(x)=2^{\lambda(1)}$ for $x\leq 1$ and for any $x\in (2^{n-1},2^{n}]$, $\Gamma(x)=2^{\Lambda(n)}$, $n\in \bn$. Clearly $\Gamma$ is non-decreasing and $\lim_{x\to \infty}\Gamma(x)=\infty$. For any pair of end points $(x_n),(y_n)\in end(X_n,\Phi_n)$ with $D((x_n),(y_n))=2^{n_0}$ the distance between their correspondent subsequences $i(x_n)=(x_{\alpha(n)}),i(y_n)=(y_{\alpha(n)})$ is exactly $D'((x_{\alpha(n)}),(y_{\alpha(n)}))=\Gamma(2^{n_0})=2^{\Lambda(n_0)}$. Then $\Gamma$ is an extension function for $i\co end(X_n,\Phi_n) \to end((X_{\alpha(n)},\tilde{\Phi}_n))$ and $i$ is bornologous. A similar argument works for the multi-map $i^{-1}$.
\end{proof}

\begin{prop} Consider two ultrametric spaces $U_1,U_2$ and $(X_n,\Phi_n), (Y_n,\Psi_n)$ their associated $D_+$--chains. Then, there is a bornologous multi-map $\Phi \co U_1\Rightarrow U_2$ if and only if there is a morphism of $D_+$--chains $(f_n,\sigma) \co (X_n,\Phi_n)\to (Y_n,\Psi_n)$.
\end{prop}

\begin{proof} If there is a bornologous multi-map $\Phi:U_1\Rightarrow U_2$, consider the map $\gamma_{\varrho_\Phi}\co \bn \to \bn$ from (\ref{gamma2}). Making $\sigma=\gamma_{\varrho_\Phi}$, $\Phi$ canonically induces the maps $f_n\co X_n\to Y_{\sigma(n)}$ and the diagram commutes.

The morphism $(f_n,\alpha)$ induces a multi-map $\Phi\co end(X_n,\Phi_n)\Rightarrow end(Y_n,\Psi_n)$ where $\Phi((x_n))$ is the set of sequences $(y_n)\in end(Y_n,\Psi_n)$ such that $y_{\sigma(n)}=f_n(x_n)$. It is immediate to check that this multi-map is bornologous, and from Lemma \ref{telescope2} together with Proposition \ref{end2}, it follows that there is a bornologous multi-map $\Phi:U_1\Rightarrow U_2$.
\end{proof}

\begin{lema}\label{alpha2} If $\Phi \co X\Rightarrow Y$ is an asymorphism then there are increasing maps $\alpha(\Phi),\beta(\Phi^{-1}) \co \bz_+ \to \bz_+$ such that $\gamma_{\varrho_\Phi}(\alpha(i))\leq \beta(i)$, $\gamma_{\varrho_{\Phi^{-1}}}(\beta(i))\leq \alpha(i+1)$ for every $i\in \bz_+$.
\end{lema}

\begin{proof} First, let $\alpha(1)=1$ and $\beta(1)=\gamma_\Phi(1)$. 

If we have defined $\alpha(i-1),\beta(i-1)$ for any $i>1$ then, it suffices make $\alpha(i)= \max\{\alpha(i-1)+1,\gamma_{\Phi^{-1}}(\beta(i-1))\}$ and $\beta(i)= \max\{\beta(i-1)+1,\gamma_{\Phi}(\alpha(i))\}$. 
\end{proof}

\begin{prop}\label{Asym} If  $(X_n,\Phi_n),(Y_n,\Psi_n)$ are two $D_+$--chains, then $(X_n,\Phi_n)\sim_{z-z}(Y_n,\Psi_n)$ if and only if 
$end(T,v)$ and $end(T',w)$ are asymorphic.
\end{prop}

\begin{proof} If there is a common zig-zag chain, the existence of an asymorphism follows immediately from Lemma \ref{telescope2}.

Now, suppose that there is an asymorphism $\Phi \co end(X_n,\Phi_n)\Rightarrow end(Y_n,\Psi_n)$. By Lemma \ref{alpha2}, if $\Phi$ is an asymorphism then there are increasing maps $\alpha(\Phi),\beta(\Phi^{-1}) \co \bn \to \bn$ such that $\gamma_{\varrho_\Phi}(\alpha(i))\leq \beta(i)$, $\gamma_{\varrho_{\Phi^{-1}}}(\beta(i))\leq \alpha_{i+1}$ for every $i\in \bn$. 

Therefore, $\Phi$ and $\Phi^{-1}$ canonically induce (as we saw in \ref{carac1}) unique surjective maps $\mathcal{V}_{2i-1}\co X_{\alpha(i)}\to Y_{\beta(i)}$ and $\mathcal{V}_{2i}\co Y_{\beta(i)}\to X_{\alpha(i+1)}$. The ball associated to the vertex $\mathcal{V}_{k+1} \circ \mathcal{V}_{k}(x)$ will contain, by construction, the ball associated to the vertex $x$, and  $\mathcal{V}_{k+1} \circ \mathcal{V}_{k}$ coincides with the map induced by the inclusion. Therefore, making $Z_{2i-1}=X_{\alpha(i)}$ and $Z_{2i}=Y_{\beta(i)}$, $(Z_n,\mathcal{V}_n)$ is a common zig-zag $D_+$--chain of $(X_n,\Phi_n),(Y_n,\Psi_n)$.
\end{proof}

From this, and Proposition \ref{end2}, it follows:

\begin{teorema} Two ultrametric spaces are coarse equivalent if and only if there is a common zig-zag chain between their associated $D_+$--chains. \hfill$\blacksquare$
\end{teorema}

Let $(X_n,\Phi_n),(Y_n,\Psi_n)$ be two $D_+$--chains and $(Z_n,\mathcal{V}_n)$ a common zig-zag $D_+$--chain with increasing maps $\alpha,\beta\co \bn \to \bn$ such that $Z_{2n-1}=X_{\alpha(n)}$ and $Z_{2n}=Y_{\beta(n)}$. Let us define $$f_Z\co  end(X_{\alpha(n)},\tilde{\Phi}_n) \to end(Y_{\beta(n)},\tilde{\Psi}_n)$$ such that for any end point $(x_{\alpha(n)})\in end(X_{\alpha(n)},\tilde{\Phi}_n)$, $f_Z((x_{\alpha(n)}))=(\mathcal{V}_{2n-1}(x_{\alpha(n)}))\in end(Y_{\beta(n)},\tilde{\Psi}_n)$.

Proposition \ref{lipschitz} is not true in the case of $D_+$--chains, since the induced map between the end spaces is not necessarily injective. There would be a bi-Lipschitz equivalence restricted to large scale. Moreover,  

\begin{prop}\label{LargeLip} Given two $D_+$--chains $(X_n,\Phi_n),(Y_n,\Psi_n)$, then $end(X_n,\Phi_n)$ and $end(Y_n,\Psi_n)$ are asymorphic if and only if there are increasing sequences $\alpha,\beta\co \bn\to \bn$ and a map   $F\co end(X_{\alpha(n)},\tilde{\Phi}_n) \to end(Y_{\beta(n)},\tilde{\Psi}_n)$ such that for any pair of end points $(x_{\alpha(n)}),(y_{\alpha(n)})\in end(X_{\alpha(n)},\tilde{\Phi}_n)$, $$D_{T(\beta)}(f((x_{\alpha(n)})),f((y_{\alpha(n)}))) \leq  D_{T(\alpha)}((x_{\alpha(n)}),(y_{\alpha(n)}))$$ and if $D_{T(\alpha)}((x_{\alpha(n)}),(y_{\alpha(n)}))$ > 2, then $$D_{T(\alpha)}((x_{\alpha(n)}),(y_{\alpha(n)})) \leq 2\cdot D_{T(\beta)}(f((x_{\alpha(n)})),f((y_{\alpha(n)}))).$$
\end{prop}

\begin{proof} If there is such a map $F$, then  $end(X_{\alpha(n)},\tilde{\Phi}_n), end(Y_{\beta(n)},\tilde{\Psi}_n)$ are, in particular, asymorphic and so they are $end(X_n,\Phi_n)$ and $end(Y_n,\Psi_n)$ by Lemma \ref{telescope2}.  

If $end(X_n,\Phi_n)$ and $end(Y_n,\Psi_n)$ are asymorphic then, by Proposition \ref{Asym}, there is a common zig-zag $D_+$--chain given by sequences $\alpha,\beta$. It is immediate to check that the map $f_Z\co end(X_{\alpha(n)},\tilde{\Phi}_n) \to end(Y_{\beta(n)},\tilde{\Psi}_n)$ holds the conditions above.
\end{proof}




\section{Uniform homeomorphisms}

In this section we treat the category $\mathcal{C}_3$. The idea is to consider only the left side of the $D$--chain but to avoid using as index set the negative integers we change the orientation of the chain and therefore the construction of $\gamma$. Also, we have to be careful with the fact that the expansion function is defined on some interval $[0,S]$. Let $\bn_{\geq n_0}:=\{n\in \bn \ | \ n\geq n_0 \}$.

Given a uniformly continuous map $\Phi$ and its expansion function $\varrho_\Phi \co [0,S]\to [0,\frac{1}{2}]$ let $n_0\in \bn$ such that $2^{-n_0}\leq S < 2^{-n_0+1}$ if $S<\frac{1}{2}$ or $n_0=1$ otherwise, and let us define $\gamma_{\varrho_\Phi}\co \bn_{\geq n_0} \to \bn$ as follows, \begin{equation}\label{gamma3} \gamma_{\varrho_\Phi}(n):=[-log_2 (\varrho_\Phi(2^{-n}))].\end{equation}

Hence, for all points $x,x'\in X$, if $d_X(x,x')\leq  2^{-n}$ then $d_Y(\Phi(x),\Phi(x'))\leq 2^{-\gamma_{\varrho_\Phi}(n)}$ and $\gamma_{\varrho_\Phi}$ is non-decreasing, $\lim_{n\to \infty}\gamma_{\varrho_\Phi}(n)=\infty$ since $\Phi$ is uniformly continuous.

Given an ultrametric space $U$ there is a $D_-$--chain associated to it. For each $n\in \bn $ let $X_n$ be the partition of $U$ in balls of radius $2^{-n}$. Let $\Phi_n\co X_{n+1}\to X_n$ the map canonically induced by the inclusion for any $n\in \bn$. $(X_n,\Phi_n)$ will be called the \emph{$D_-$--chain associated to $U$}. Conversely, given a $D_-$--chain we can obtain an ultrametric space as follows.

The end space is then:

$$end(X_n,\Phi_n):=\{(x_n)_{n\in \bn} \ | \ x_n\in X_n \mbox{ and } \Phi(x_{n+1})=x_{n}\},$$ and let the metric be $$D((x_n),(y_n))= \left\{ \begin{tabular}{l} $2^{-n_0} \mbox{ if there is } n_0=\max\{n: x_n=y_n \},$\\

$1 \mbox{ if } x_n\neq y_n \ \forall \, n.$\end{tabular}
 \right.$$
Clearly, $D$ is an ultrametric. 

\begin{prop}\label{end3} If $(U,d)$ is a complete ultrametric space and $(X_n,\Phi_n)$ is the $D_-$--chain associated to $U$, then $(U,d)$ and $end(X_n,\Phi_n)$ are uniformly homeomorphic. 
\end{prop}

\begin{proof} First, note that if $(U,d)$ is complete there is a bijection $i\co U\to end(X_n,\Phi_n)$.

Notice that, by the properties of the ultrametric, $d(x,y)\leq  2^{-n}$ if and only if the points are in the same ball in the partition $X_n$. Hence, if $d(x,y)\leq  2^{-n}$, $d(x,y)\leq D(i(x),i(y))\leq 2d(x,y)$ and $i$ is a uniform homeomorphism.
\end{proof}

In the case of uniform maps, we need to consider in the description of the morphisms of $D_-$ chains the radius such that the image of the ball will have diameter bounded by 1/2 (i.e. the interval $[0,S]$ on which the expansion function is defined): 

\begin{definicion} A \emph{morphism of $D_-$-chains} $(f_n,\sigma,n_0) \co (X_n,\Phi_n)\to (Y_n,\Psi_n)$ consists of a natural number $n_0$, a non-decreasing function $\sigma\co \bn_{\geq n_0} \to \bn$ such that $\lim_{n\to \infty}\sigma(n)=\infty$, and maps $f_n\co X_n \to Y_{\sigma(n)}$ \ $\forall n\geq n_0$ such that the following diagram commutes:
\begin{small}$$
\xymatrix{  \longleftarrow  & X_{n} \ar[dl]_{f_{n}} & \longleftarrow   &  X_{n+1}  \ar[dl]_{f_{n+1}}  &  \longleftarrow   & X_{n+2}\ar[dl]_{f_{n+2}}   \\
           Y_{\sigma(n)}  \longleftarrow  & \cdots  &  \longleftarrow  Y_{\sigma(n+1)}  \longleftarrow  & \cdots  &  \longleftarrow   Y_{\sigma(n+2)} &  \longleftarrow    }
$$\end{small}
\end{definicion}

\begin{lema} \label{telescope3} If $(X_n,\Phi_n)$ is a $D_-$--chain, $\alpha \co \bn \to \bn$ is an increasing map  and $(X_{\alpha(n)},\tilde{\Phi}_n)$ is the $\alpha$-subchain, then $end(X_n,\Phi_n)$ is uniformly homeomorphic to $end(X_{\alpha(n)},\tilde{\Phi}_n)$.
\end{lema}

\begin{proof} Consider the canonical map $i\co end(X_n,\Phi_n)\to end(X_{\alpha(n)},\tilde{\Phi}_n)$. 

Let us define the function $\Lambda\co \bn \to \bn$ such that $\Lambda(n)=\max\{k:\alpha(k)\leq n\}$. Since $\alpha$ is increasing, it follows that $\Lambda$ is non-decreasing and $\lim_{n\to \infty}\Lambda(n)=\infty$. Now, consider the function $\Gamma\co [0,\infty) \to [0,\infty)$ such that $\Gamma(S)=\frac{1}{2}$ and for any $x\in [S\cdot 2^{-n},S\cdot 2^{-n+1})$, $\Gamma(x)=2^{-\Lambda(n)}$. Clearly $\Gamma$ is non-decreasing, $\lim_{x\to 0}\Gamma(x)=0$. For any two  end points $(x_n),(y_n)\in end(X_n,\Phi_n)$ with $D((x_n),(y_n))=2^{-n_0}$ the distance between their correspondent subsequences $i(x_n)=(x_{\alpha(n)}),i(y_n)=(y_{\alpha(n)})$ is exactly $D'((x_{\alpha(n)}),(y_{\alpha(n)}))=\Gamma(2^{-n_0})=2^{-\Lambda(n_0)}$. Then $\Gamma$ is an extension function for $i\co end(X_n,\Phi_n) \to end((X_{\alpha(n)},\tilde{\Phi}_n))$ and $i$ is a uniform homeomorphism. A similar argument works for $i^{-1}$.
\end{proof}

\begin{prop}  Consider two complete ultrametric spaces $U_1,U_2$ and let $(X_n,\Phi_n)$, $(Y_n,\Psi_n)$ be their associated $D_-$--chains. Then, there is a uniformly continuous map $\Phi \co U_1\to U_2$ if and only if there is a morphism of $D_-$--chains $(f_n,\alpha,n_0) \co (X_n,\Phi_n)\to (Y_n,\Psi_n)$.
\end{prop}

\begin{proof} If there is a uniformly continuous map $\Phi:U_1\to U_2$, consider the map $\gamma_{\varrho_\Phi}\co \bn \to \bn$ from (\ref{gamma3}). Making $\sigma=\gamma_{\varrho_\Phi}$, $\Phi$ canonically induces the maps $f_n\co X_n\to Y_{\sigma(n)}$ $\forall n\geq n_0$ and the diagram commutes.

The morphism $(f_n,\alpha,n_0)$ induces a map $F\co end(X_n,\Phi_n)\to end(Y_n,\Psi_n)$ where $F((x_n))$ is the unique sequence $(y_n)\in end(Y_n,\Psi_n)$ such that $y_{\sigma(n)}=f_n(x_n) \ \forall n\geq n_0$. It is immediate to check that this map is uniformly continuous, and from Lemma \ref{telescope3} together with Proposition \ref{end3}, it follows that there is a uniformly continuous map $\Phi:U_1\to U_2$.
\end{proof}

\begin{lema}\label{alpha3} If $\Phi \co X\to Y$ is a uniform homeomorphism then there are increasing maps
$\alpha(\Phi),\beta(\Phi^{-1}) \co \bn \to \bn$ such that $\gamma_{\varrho_{\Phi^{-1}}}(\beta(i))\leq \alpha(i)$, $\gamma_{\varrho_\Phi}(\alpha(i+1))\leq \beta(i)$ for every $i\in \bn$.
\end{lema}

\begin{proof} First, let $\alpha(1)=1$. 

Since $lim_{n\to \infty}\gamma_{\Phi^{-1}}(n)=\infty$ for any $\alpha(i)$ there exist some $n_{\alpha(i)}>0$ such that for any $n\geq n_{\alpha(i+1)}$, $\gamma_{\Phi^{-1}}(n)\geq \alpha(i)$.

Let $\beta(1)=n_{\alpha(1)}$.

Since $lim_{n\to \infty}\gamma_\Phi(n)=\infty$ for any $\beta(i)$ there exist some $n_{\beta(i)}$  such that for any $n\geq n_{\beta(i)}$, $\gamma_{\Phi}(n)\geq \beta(i)$.

If we have defined $\alpha(i-1),\beta(i-1)$, then, it suffices make $\alpha(i)=\max\{\alpha(i-1)+1,n_{\beta(i-1)}\}$ and $\beta(i)= \max\{\beta(i-1)+1,n_{\alpha(i)}\}$.
\end{proof}

\begin{prop}\label{carac3} If  $(X_n,\Phi_n),(Y_n,\Psi_n)$ are two $D_-$--chains, then $(X_n,\Phi_n)\sim_{z-z}(Y_n,\Psi_n)$ if and only if $end(T,v), end(T',w)$ are uniform homeomorphic.
\end{prop}

\begin{proof} If there is a common zig-zag $D_-$--chain, the existence of an asymorphism follows immediately from Lemma \ref{telescope3}.

Now, suppose there is a uniform homeomorphism $\Phi \co end(X_n,\Phi_n)\to end(Y_n,\Psi_n)$. Then, by Lemma \ref{alpha3},  there are increasing maps
$\alpha(\Phi),\beta(\Phi^{-1}) \co \bn \to \bn$ such that $\gamma_{\varrho_{\Phi^{-1}}}(\beta(i))\leq \alpha(i)$, $\gamma_{\varrho_\Phi}(\alpha(i+1))\leq \beta(i)$ for every $i\in \bn$.

Therefore, $\Phi$ and $\Phi^{-1}$ induce respectively unique surjective maps $\mathcal{V}_{2i-1}\co Y_{\beta(i)} \to X_{\alpha(i)}$ and $\mathcal{V}_{2i}\co  X_{\alpha(i+1)} \to Y_{\beta(i)}$.  $\mathcal{V}_{k+1} \circ \mathcal{V}_{k}$ coincides with the map induced by the inclusion, and hence, making $Z_{2i-1}=X_{\alpha(i)}$ and $Z_{2i}=Y_{\beta(i)}$, $(Z_n,\mathcal{V}_n)$ is a common zig-zag $D_-$--chain of $(X_n,\Phi_n),(Y_n,\Psi_n)$.
\end{proof}

From this, and Proposition \ref{end3}, it follows:

\begin{teorema} Two complete ultrametric spaces $U_1,U_2$ are uniformly homeomorphic if and only if there is a common zig-zag $D_-$--chain between their associated $D_-$--chains. \hfill$\blacksquare$
\end{teorema}

\begin{definicion} A function between metric spaces $f\co X\to Y$ is \emph{small scale bi-Lipschitz} if there is a constant $K>0$ and a real number $\epsilon >0$ such that for any pair of points $x,x'\in X$ with $d(x,x')<\epsilon$, $\frac{1}{K} \cdot d_X(x,x') \leq d_y(f(x),f(x')) \leq K\cdot d_X(x,x')$. In there is such a map, we say that $X,Y$ are \emph{small scale bi-Lipschitz equivalent}.
\end{definicion}

Let $(X_n,\Phi_n),(Y_n,\Psi_n)$ be two $D_-$--chains and $(Z_n,\mathcal{V}_n)$ a common zig-zag $D_-$--chain with increasing maps $\alpha,\beta\co \bn \to \bn$ such that $Z_{2n-1}=X_{\alpha(n)}$ and $Z_{2n}=Y_{\beta(n)}$. Let us define $$f_Z\co  end(X_{\alpha(n)},\tilde{\Phi}_n) \to end(Y_{\beta(n)},\tilde{\Psi}_n)$$ such that for any end point $(x_{\alpha(n)})\in end(X_{\alpha(n)},\tilde{\Phi}_n)$, $f_Z((x_{\alpha(n)}))$ is the unique end point $(y_{\beta(n)})\in end(Y_n,\Psi_n)$ such that $\mathcal{V}_{2n-2}(x_{\alpha(n)})=y_{\beta(k-1)}$ $\forall n\geq 2$.

\begin{prop} Given two $D_-$--chains $(X_n,\Phi_n),(Y_n,\Psi_n)$, then $end(X_n,\Phi_n)$ and $end(Y_n,\Psi_n)$ are uniformly homeomorphic if and only if there are increasing sequences $\alpha,\beta\co \bn\to \bn$ such that $end(X_{\alpha(n)},\tilde{\Phi}_n)$ and $end(Y_{\beta(n)},\tilde{\Psi}_n)$ are small scale bi-Lipschitz equivalent.
\end{prop}

\begin{proof} If $end(X_{\alpha(n)},\tilde{\Phi}_n)$ and $end(Y_{\beta(n)},\tilde{\Psi}_n)$ are small scale bi-Lipschitz equivalent, then they are, in particular uniformly homeomorphic and so they are $end(X_n,\Phi_n)$ and $end(Y_n,\Psi_n)$ by Lemma \ref{telescope3}.  

If $end(X_{\alpha(n)},\tilde{\Phi}_n)$ and $end(Y_{\beta(n)},\tilde{\Psi}_n)$ are uniformly homeomorphic then, by Proposition \ref{carac3}, there is a common zig-zag $D_-$--chain given by sequences $\alpha,\beta$. Then, the map $f_Z\co end(X_{\alpha(n)},\tilde{\Phi}_n) \to end(Y_{\beta(n)},\tilde{\Psi}_n)$ described above, for any pair of end points $(x_{\alpha(n)}),  (x'_{\alpha(n)})\in  end(X_{\alpha(n)},\tilde{\Phi}_n)$ with  $D_{T(\alpha)}((x_{\alpha(n)}),(x'_{\alpha(n)})) \leq \frac{1}{4}$, holds that
$$D_{T(\alpha)}((x_{\alpha(n)}),(x'_{\alpha(n)})) \leq  D_{T(\beta)}(f_Z((x_{\alpha(n)})),f_Z((x'_{\alpha(n)}))) \leq 2\cdot D_{T(\alpha)}((x_{\alpha(n)}),(x'_{\alpha(n)}))$$ and it is small scale bi-Lipschitz.
\end{proof}

\begin{cor} Two ultrametric spaces $U_1,U_2$ are uniformly homemorphic if and only if there are increasing sequences $\alpha,\beta\co \bn\to \bn$ such that $end(X_{\alpha(n)},\tilde{\Phi}_n)$ and $end(Y_{\beta(n)},\tilde{\Psi}_n)$ are small scale bi-Lipschitz equivalent. \hfill$\blacksquare$
\end{cor}

Let $(U,d)$ be an ultrametric space and $(n_i), i> 0$ an increasing sequence of numbers. Let us define $f_{(n_i)} \co [0,\infty) \to \bn$ a non-decreasing function such that for any $t\geq 2^{-n_1}$, $f(t)=2^{-1}$ and for any $t\in (2^{-n_{i+1}},2^{-n_i}]$ $f(t)=2^{-i}$ for every $i>1$. Let us denote simply by $(U,d(n_i))$ the ultrametric space $(U,d_{f_{(n_i)}})$ which depends only on the original ultrametric and the sequence $(n_i)$.

\begin{cor} Two ultrametric spaces $(U_1,d_1),(U_2,d_2)$ are uniformly homeomorphic if and only if there are increasing sequences of numbers $(n_i),(m_i)$ such that  
$(U_1,d_1(n_i))$ and $(U_2,d_2(m_i)$ are small scale bi-Lipschitz equivalent. \hfill$\blacksquare$
\end{cor}

\section{Towers and admissible morphisms}

In \cite{BZ}, Taras Banakh and Ihor Zarichnyy give a classification of ultrametric spaces up to coarse geometry. They prove their results by induction on partially ordered sets called towers. The following definitions are stated as they appear in their paper. 

A partially ordered set $T$ is a \emph{tree} if $T$ has the smallest element and for every point $x\in T$ the lower cone $\downarrow x$ is well-ordered. By the \emph{lower cone} (resp. \emph{upper cone}) of a point $x$ of a partially ordered set $T$ we understand the set $\downarrow x=\{y\in T:y\leq x\}$ (resp. $\uparrow x=\{y\in T:y\geq x\}$). A subset $A$ will be called a \emph{lower} (resp. \emph{upper}) \emph{set} if $\downarrow a \subset A$ (resp. $\uparrow a \subset A$) for every $a\in A$. A partially ordered set $T$ is \emph{well-founded} if each subset $A\subset T$ has a minimal element $a\in A$. The minimality of $a$ means that each point $a'\in A$ with $a'\leq a$ is equal to $a$. By $\min T$ we shall denote the set of all minimal elements of $T$.

\begin{definicion} A partially ordered set $T$ is called a \emph{tower} if 
\begin{itemize}
\item[(1)] $T$ is well-founded;
\item[(2)] any two elements $x,y \in T$ have the smallest upper bound $\sup\{x,y\}$ in $T$;
\item[(3)] for any $x\in T$ the upper cone $\uparrow x$ is linearly ordered;
\item[(4)] for any point $a\in T$ there is a finite number $n=lev_T(a)$ such that for every minimal element $x\in \downarrow a$ of $T$ the order interval $[x,a]=\uparrow x \cap \downarrow a$ has cardinality $|[x,a]|=n$.
\end{itemize}
\end{definicion}

The function $lev_T\co T\to \bn, \ lev_T\co a \mapsto lev_T(a)$, from the last item is called the \emph{level function}. 

The level function $lev_T\co T\to \bn$ divides $T$ into the levels $L_i=lev_T^{-1}(i), \ i\in \bn$. The level $L_1=min T$ is called \emph{the base} of $T$ and denoted by $[T]$.

Each tower carries a canonic \emph{path metric} $d_T$ defined by the formula $$d_T(x,y)=2\cdot lev_T(sup(x,y))-(lev_T(x)+lev_T(y)) \mbox{ for } x,y\in T.$$
The path metric restricted to the base  $[T]$ of $T$ is an ultrametric.

Given a tower $T$ with levels $L_i$, we can define a $D_+$--chain $(L_i,\Phi_i)$ with $\Phi\co L_i\to L_{i+1}$ such that $\Phi(x_i)=x_{i+1}$ for any $x_i \leq x_{i+1}$.

\begin{prop}\label{ends} For any tower $T$ with levels $L_i$, $end(L_i,\Phi_i)$ is coarse equivalent to $[T]$.
\end{prop}

\begin{proof} This is readily seen since $D(x,y)=2^n$ if and only if $d_T(x,y)=2n$.  
\end{proof}

For every point $x\in T$ of a tower $T$, the set $L_i\cap \downarrow x$ with $i=lev_T(x)-1$ is denoted $pred(x)$ and it is called the set of \emph{parents of $x$}. 

\begin{definicion} Let $T_1,T_2$ be two towers. A map $\phi\co A \to T_2$ defined on a lower subset $A=\downarrow A$ of $T_1$ is called an \emph{admissible morphism} if
\begin{itemize}\item[(1)] $lev(\phi(a))=lev(a)$ for all $a\in A$;
\item[(2)] $a\leq a'$ in $A$ implies $\phi(a)\leq \phi(a')$;
\item[(3)] $\phi(a)=\phi(a')$ for $a,a'\in A$ implies that $a,a'\in pred(v)$ for some $v\in T_1$;
\item[(4)] $\phi(A)$ is a lower subset of $T_2$;
\item[(5)] $ |\phi(max \, A)|\leq 1$, where $max \, A$ stands for the (possibly empty) set of maximal elements of the domain $A$.
\end{itemize}
\end{definicion}

As we mentioned in the introduction, see \ref{Admissible2}, Lemma 2 in \cite{BZ} states that the restriction to the base of an admissible morphism between towers is an asymorphism. Using $D_+$--chains we proof that this is in fact an if and only if condition.

Consider two towers $T,T'$ and their corresponding $D_+$--chains $(X_n,\Phi_n),(Y_n,\Psi_n)$. Let $(Z_n,\mathcal{V}_n)$ be a common zig-zag $D_+$--chain for $(X_n,\Phi_n),(Y_n,\Psi_n)$ with increasing maps $\alpha ,\beta\co \bn\to \bn$ such that $Z_{2n-1}=X_{\alpha(n)}$ and $Z_{2n}=X_{\beta(n)}$. $(X_{\alpha(n)},\tilde{\Phi}_n),(Y_{\beta(n)},\tilde{\Psi}_n)$ define subchains. Let $T(\alpha),T'(\beta)$ be the corresponding subtowers of $T,T'$ defined respectively by levels $\alpha(n)$ and $\beta(n)$, $n\in \bn$ and let $T(\alpha(n))$ denote $lev_{T(\alpha)}^{-1}(n)$ (its $n^{th}$ level).

Let $f_Z\co  T(\alpha) \to T'(\beta)$ be such that for every $n$, $f_Z|_{T(\alpha(n))}=\mathcal{V}_{2n-1}\co T(\alpha(n)) \to T'(\beta(n))$.

It is immediate to check the following:

\begin{prop}\label{Admissible1} Given a common zig-zag $D_+$--chain $(Z_n,r_n)$ for two towers $T,T'$, $f_Z$ is an admissible map. \hfill$\blacksquare$
\end{prop}

From Proposition \ref{Admissible1} together with \ref{Asym} and \ref{ends}, and Proposition \ref{Admissible2} we conclude that

\begin{cor}\label{Admiss} Given two towers $T_1,T_2$, $[T_1]$ and $[T_2]$ are asymorphic if and only if there is an admissible map $f\co T(\alpha) \to T'(\beta)$ for some pair of sequences $\alpha,\beta$. \hfill$\blacksquare$
\end{cor}

What follows is a version of \ref{LargeLip} for the metric given here to the base.

\begin{prop} Given two towers $T_1,T_2$, $[T_1]$ and $[T_2]$ are asymorphic if and only if there are increasing sequences $\alpha,\beta\co \bn\to \bn$ such that $[T_1(\alpha)]$ and $[T_2(\beta)]$ are roughly isometric.
\end{prop}

\begin{proof} If $[T_1(\alpha)]$ and $[T_2(\beta)]$ are roughly isometric, then they are, in particular asymorphic.  Proposition \ref{ends} and Lemma \ref{telescope2} yield that $[T_1]$ and $[T_2]$ are asymorphic.

From Corollary \ref{Admiss}, we obtain an admissible map $f\co T(\alpha) \to T'(\beta)$. For any pair of points $x,y\in [T_1(\alpha)]$, condition (3) in the definition of admissible map implies that $d_{T(\beta)}(f(x),f(y)) \leq  d_{T(\alpha)}(x,y) \leq d_{T(\beta)}(f(x),f(y))+2$.
\end{proof}

\end{document}